\documentclass[12pt]{amsart}

\newtheorem{theorem}{Theorem}[section]
\newtheorem{corollary}[theorem]{Corollary}
\newtheorem{remark}[theorem]{Remark}
\newtheorem{lemma}[theorem]{Lemma}

\numberwithin{equation}{section}
\author{Xiaoli Han, Jiayu Li, Jun Sun}

\address{Xiaoli Han, Department of Mathematical Sciences, Tsinghua University \\ Beijing 100084, P. R. of China.}
\email{xlhan@math.tsinghua.edu.cn}

\address{Jiayu Li, School of Mathematical Sciences, University of Science and Technology of China Hefei 230026 \\ AMSS CAS Beijing 100190, P. R. China}
\email{jiayuli@ustc.edu.cn}

\address{Jun Sun, School of Mathematical Sciences\\
Wuhan University\\ Wuhan 430072, P. R. of China.}
\email{sunjun@whu.edu.cn}

\keywords{$\beta$-symplectic critical surfaces, compactness, bubble, tangent cone.}

\thanks {The research was supported by the National Natural Science Foundation of China,  No.11426236, No.11131007, No.11471014., No. 11401440.}

\begin{document}

\title[The deformation of symplectic critical surfaces]
{The deformation of symplectic critical surfaces in a K\"ahler
surface-II---Compactness}

\begin{abstract}
In this paper we consider the compactness of $\beta$-symplectic critical surfaces in a K\"ahler surface.
Let $M$ be a compact K\"ahler surface and $\Sigma_i\subset M$ be a sequence of closed $\beta_i$-symplectic critical surfaces with $\beta_i\to\beta_0\in (0,\infty)$. Suppose the quantity $\int_{\Sigma_i}\frac{1}{\cos^q\alpha_i}d\mu_i$ (for some $q>4$) and the genus of $\Sigma_{i}$ are bounded,
then there exists a finite set of points ${\mathcal S}\subset M$ and a subsequence $\Sigma_{i'}$ that converges uniformly in the $C^l$ topology (for any $l<\infty$) on compact subsets of $M\backslash {\mathcal S}$ to a $\beta_0$-symplectic critical surface $\Sigma\subset M$, each connected component of $\Sigma\setminus {\mathcal S}$ can be extended smoothly across ${\mathcal S}$.
\end{abstract}

\maketitle

{\bf Mathematics Subject Classification (2010):} 53C42 (primary), 58J05 (secondary).

\section{Introduction}

\allowdisplaybreaks

Let $(M,\omega,J,\bar g)$ be a K\"ahler surface with K\"ahler form $\omega$, compatible complex structure $J$ and associated Riemannian metric $\bar g$. Then we have
$$
\bar g(U,V)=\omega(U,JV).
$$
For a compact oriented real surface $\Sigma$ which is smoothly
immersed in $M$, one defines, following \cite {CW}, the K\"ahler
angle $\alpha$ of $\Sigma$ in $M$ by
\begin{equation}\label{e1}\omega|_\Sigma=\cos\alpha d\mu_\Sigma\end{equation} where $d\mu_\Sigma$
is the area element of $\Sigma$ of the induced metric from
$\bar g$. We say that $\Sigma$ is a {\em holomorphic curve} if
$\cos\alpha \equiv 1$, $\Sigma$ is a {\em Lagrangian surface} if
$\cos\alpha \equiv 0$ and $\Sigma$ is a {\em symplectic surface} if
$\cos\alpha > 0$.

In \cite{HL} we considered the functional
$$L=\int_\Sigma\frac{1}{\cos\alpha}d\mu.$$
The Euler-Lagrange equation of this functional is
$$\cos^3\alpha {\bf{H}}=(J(J\nabla\cos\alpha)^\top)^\bot.$$ We called
such a surface a {\em symplectic critical surface}. We studied the
properties of the symplectic critical surfaces in \cite{HL} and also examined the second variation formula of the functional $L$ in \cite{HL2}.

In order to approach the problem of the existence of holomorphic curves in a K\"ahler surface, we (\cite{HLS1}) studied a family of functionals
$$L_\beta=\int_\Sigma\frac{1}{\cos^\beta\alpha}d\mu.$$
The critical points of the functional $L_\beta$ in the class of
symplectic surfaces in a K\"ahler surface are called
{\em $\beta$-symplectic critical surfaces}.  The
Euler-Lagrange equation of $L_\beta$ is

\begin{equation}\label{me}\cos^3\alpha
{\bf{H}}-\beta(J(J\nabla\cos\alpha)^\top)^\bot=0, \end{equation}
where ${\bf H}$ is the mean curvature vector of $\Sigma$ in $M$,
$()^\top$ means tangential components of $()$, and $()^\bot$ means the
normal components of $()$. In \cite{HLS1}, we checked that it
is an elliptic system module tangential diffeomorphisms if $\beta\geq 0$.
We also saw that $\beta$-symplectic critical surfaces share many properties with
minimal surfaces (c.f. \cite{CW}, \cite{HL}, \cite{HL2}, \cite{MW},
\cite{W1},\cite{W2}).  We (\cite{HLS1}) constructed explicit $\beta$-symplectic critical surfaces in $\mathbb{C}^2$ ($0\leq\beta <\infty$) and showed that
it converges to a plane as $\beta\to \infty$. We believe that {\em the $\beta$-symplectic critical surfaces connect
the stable minimal surface ($\beta=0$) and the holomorphic curve ($\beta\to\infty$) in a K\"ahler surface.}
In order to proceed the idea, we define the set
\begin{equation*}
    S:=\{\beta\in[0,\infty)\mid \exists\ a\ strictly\ stable\ \beta-symplecitc\ critical\
    surface\}.
\end{equation*}
In \cite{HLS1},  applying the second variation formula and an implicit function theorem due to B.White (\cite{White}), we proved that: {\em The set $S$ is open in $[0,\infty)$.}

Of course, the proof that $S$ is closed is much more subtle. We need uniform estimates for various geometric quantities.
In this paper, as a first attempt to the closeness, we will consider the compactness of $\beta$-symplectic critical surfaces in a K\"ahler surface. Our main theorem is as follows:

\vspace{.1in}

\noindent \textbf{Main Theorem}
{\em Let $M$ be a closed K\"ahler surface and $\Sigma_i\subset M$ a sequence of closed $\beta_i$-symplectic critical surfaces with $\beta_i\to\beta_0\in (0,\infty)$,

\begin{equation}\label{EE3.3}
   \int_{\Sigma_i}\frac{1}{\cos^q\alpha_i}\leq C_1,
\end{equation}
for some $q>4$, and $g(\Sigma_{\beta_i})\leq g_0$, where $g(\Sigma_{\beta_i})$ is the genus of $\Sigma_{\beta_i}$. Then there exists a finite set of points ${\mathcal S}\subset M$ and a subsequence $\Sigma_{i'}$ that converges uniformly in the $C^l$ topology (for any $l<\infty$) on compact subsets of $M\backslash {\mathcal S}$ to a $\beta_0$-symplectic critical surface $\Sigma\subset M$. The subsequence also converges to $\Sigma$ in (extrinsic) Hausdorff distance.

Furthermore, around each singular point $p_{\gamma}\in {\mathcal S}$, there is a bubble which is a smooth holomorphic curve in ${\mathbb C}^2$. The tangent cone of $\Sigma\cup{\mathcal S}$ at $p_{\gamma}$ is a flat cone consists of union of planes in ${\mathbb C}^2$ which intersects at only one point.

Moreover, for each connected component $\Omega$ of $\Sigma$ with $p_{\gamma}\in{\mathcal S}\cap \bar\Omega$, $\Omega\cup \{p_{\gamma}\}$ is a smooth $\beta_0$-symplectic critical surface of $M$.
}

\vspace{.1in}

To prove the compactness, we first need to establish the small energy regularity theorem for $\beta$-symplectic critical surfaces. In order to show the convergence is in the sense of Hausdorff and the bubble is holomorphic, we need a uniform estimate for the lower bound of $\cos\alpha$ and the upper bound of the total curvature. For this purpose, we apply the Moser's iteration to the elliptic equation satisfied by $\cos\alpha$. The assumption (\ref{EE3.3}) is just used to start the iteration.

The following sections are organized as follows: In Section 2, we prove the  small energy regularity theorem; in Section 3, we provide the uniform lower bound for $\cos\alpha$ and the uniform upper bound of the total curvature; in Section 4, we prove the main theorem. In the appendix, we prove the monotonicity formula needed in the proof of the main theorem.

\vspace{.1in}

\section{Small Energy Regularity Theorem}

In this section, we will prove the small energy regularity theorem for $\beta$-symplectic critical surfaces, which states that a point is a regular point if there is no energy concentration. The idea follows from that of minimal submanifolds (\cite{An}), and the key point is the following compactness theorem:

\begin{theorem}\label{thm-compact}
Let $\{\Sigma_{\beta_i}\}$ be a sequence of connected $\beta_i$-symplectic critical surface in $B^4(r_0)\subset M$ with $\beta_i\to \beta_0\in[0,\infty)$ and $\cos\alpha_i\geq\delta>0$, such that $\partial \Sigma_{\beta_i}\subset\partial B^4(r_0)$. Suppose that there is a uniform constant $C_0>0$, such that
\begin{equation}\label{E3.8}
    \sup_{\Sigma_{\beta_i}}|\textbf{A}_{\beta_i}|(x)\leq C_0,
\end{equation}
for all $i$. Then there is a subsequence of $\Sigma_{\beta_i}$,
 still denote by $\Sigma_{\beta_i}$, that converges in the $C^{\infty}_{loc}$ topology to a smooth $\beta_0$-symplectic critical surface $\Sigma_{\beta_0}$ in $B^4(r_0)$
  with $\sup_{\Sigma_{\beta_0}}|\textbf{A}_{\beta_0}|(x)\leq C_0$.
\end{theorem}

\textbf{Proof:} Without loss of generality, we will prove the theorem for the case $M={\mathbb R}^4$ (see the next remark). By our assumption (\ref{E3.8}), we know that there exists a constant $\varepsilon_0>0$, depending only on $C_0$, such that each connected component of $\Sigma_{\beta_i}\cap B_{p_i}(\varepsilon)$, with $p_i\in \Sigma_{\beta_i}$, $B_{p_i}(\varepsilon)\subset\subset B^4(r_0)$ and $\varepsilon<\varepsilon_0$, can be graphed over $T_{p_i}\Sigma_{\beta_i}\cap B_{p_i}(\varepsilon)$ by a pair of functions $(f_i,g_i)$. Furthermore, the bound in (\ref{E3.8}) also gives us a uniform bound on the $C^{1,\alpha}$-norm of $(f_i,g_i)$. By direct calculation, we can see that the equation satisfied by $\beta$-symplectic critical surfaces (\ref{me}) is equivalent to (with $(f,g,\beta)=(f_i,g_i,\beta_i)$)
\begin{eqnarray}\label{E3.9}
              &f_{xx}[g^{11}g^{33}c^{2}-\beta(g^{11}g^{22}-g^{12}g^{12})(g^{34}a+g^{33}b)(g_{12}a-g_{22}b)] \nonumber\\
              &+f_{xy}[2g^{33}g^{12}c^{2}-\beta(g^{11}g^{22}-g^{12}g^{12})(g^{34}a+g^{33}b)(g_{12}b-g_{11}a) \nonumber\\
                 &\hspace{2pt} -\beta(g^{11}g^{22}-g^{12}g^{12})(g^{33}a-g^{34}b)(g_{12}a-g_{22}b)] \nonumber\\
              &+f_{yy}[g^{22}g^{33}c^{2}-\beta(g^{11}g^{22}-g^{12}g^{12})(g^{33}a-g^{34}b)(-g_{11}a+g_{12}b)] \nonumber\\
              &+g_{xx}[g^{11}g^{34}c^{2}-\beta(g^{11}g^{22}-g^{12}g^{12})(g^{34}a+g^{33}b)(-g_{22}a-g_{12}b)] \nonumber\\
              &+g_{xy}[2g^{34}g^{12}c^{2}-\beta(g^{11}g^{22}-g^{12}g^{12})(g^{34}a+g^{33}b)(g_{11}b+g_{12}a) \nonumber\\
                 &\hspace{2pt} -\beta(g^{11}g^{22}-g^{12}g^{12})(g^{33}a-g^{34}b)(-g_{22}a-g_{12}b)] \nonumber\\
              &+g_{yy}[g^{22}g^{34}c^{2}-\beta(g^{11}g^{22}-g^{12}g^{12})(g^{33}a-g^{34}b)(g_{11}b+g_{12}a)]=0,
\end{eqnarray}
and
\begin{eqnarray}\label{E3.10}
               &f_{xx}[g^{11}g^{34}c^{2}-\beta(g^{11}g^{22}-g^{12}g^{12})(g^{44}a+g^{34}b)(g_{12}a-g_{22}b)] \nonumber\\
              &+f_{xy}[2g^{34}g^{12}c^{2}-\beta(g^{11}g^{22}-g^{12}g^{12})(g^{44}a+g^{34}b)(g_{12}b-g_{11}a) \nonumber\\
                 &\hspace{2pt} -\beta(g^{11}g^{22}-g^{12}g^{12})(g^{34}a-g^{44}b)(g_{12}a-g_{22}b)] \nonumber\\
              &+f_{yy}[g^{22}g^{34}c^{2}-\beta(g^{11}g^{22}-g^{12}g^{12})(g^{34}a-g^{44}b)(-g_{11}a+g_{12}b)] \nonumber\\
              &+g_{xx}[g^{11}g^{44}c^{2}-\beta(g^{11}g^{22}-g^{12}g^{12})(g^{44}a+g^{34}b)(-g_{22}a-g_{12}b)] \nonumber\\
              &+g_{xy}[2g^{44}g^{12}c^{2}-\beta(g^{11}g^{22}-g^{12}g^{12})(g^{44}a+g^{34}b)(g_{11}b+g_{12}a) \nonumber\\
                 &\hspace{2pt} -\beta (g^{11}g^{22}-g^{12}g^{12})(g^{34}a-g^{44}b)(-g_{22}a-g_{12}b)] \nonumber\\
              &+g_{yy}[g^{22}g^{44}c^{2}-\beta(g^{11}g^{22}-g^{12}g^{12})(g^{34}a-g^{44}b)(g_{11}b+g_{12}a)]=0,
\end{eqnarray}
where
\begin{equation*}
    a=f_{y}+g_{x}, b=f_{x}-g_{y}, c=1+f_{x}g_{y}-f_{y}g_{x}=\sqrt{det(g)} \cos\alpha.
\end{equation*}
For the derivation of (\ref{E3.9}) and (\ref{E3.10}), we refer to the proof of Theorem 2.3 of \cite{HL}. By Proposition 3.1 of \cite{HLS1} (see also Theorem 2.3 of \cite{HL}), we see that the system (\ref{E3.9}) and (\ref{E3.10}) is strictly elliptic, with elliptic constant depending on $\beta_0$, $C_0$ and $\delta$. Schauder estimate for elliptic systems gives us the uniform $C^{k,\alpha}$ bound for $(f_i,g_i)$, which implies the existence of a convergent subsequence of $(f_i,g_i)$ in the $C^{\infty}_{loc}$ topology to a solution $(f_{\infty},g_{\infty})$ to (\ref{E3.9}) and (\ref{E3.10}) with $\beta=\beta_0$. By a diagonal argument, we see that there exists a subsequence of ${\Sigma_{\beta_i}}$, which converges in the $C^{\infty}_{loc}$ topology to a smooth surface $\Sigma_{\beta_0}$ in $B^4(r_0)$ and $\Sigma_{\beta_0}$ has the required properties.
\hfill Q.E.D.

\begin{remark}\label{Rmk6.2}
When the ambient manifold is a K\"ahler surface $M$, there will be some extra terms in the system (\ref{E3.9}) and (\ref{E3.10}), which only involve the first order of $f$ and $g$. Therefore, the argument above remains applicable.
\end{remark}

\begin{theorem}\label{thm-small}
Let $0\leq A<B<\infty$. There exist constants $\varepsilon>0$ and $\rho>0$ (depending on $M$), such that if $r_0\leq\rho$, $\Sigma$ is a $\beta$-symplectic critical surface with $\partial\Sigma_{\beta}\subset\partial B^4(r_0)$, $\beta\in [A,B]$, $\cos\alpha\geq\delta>0$, and
\begin{equation}\label{E3.11}
    \int_{B_{r_0}\cap\Sigma}|\textbf{A}|^2d\mu\leq \varepsilon,
\end{equation}
then we have
\begin{equation}\label{E3.12}
    \max_{\sigma\in[0,r_0]}\sigma^2\sup_{B_{r_0-\sigma}\cap \Sigma}|\textbf{A}|^2\leq 4.
\end{equation}
\end{theorem}

\textbf{Proof:} As before, we again assume that $M={\mathbb R}^4$. We will follow Anderson's argument for minimal surfaces (\cite{An}, \cite{CM}). We prove it by contradiction. If (\ref{E3.12}) is false, then, there will be a sequence of $\beta_i$-symplectic crtical surface $\Sigma_{\beta_i}$, given by the immersion $F_i:\Sigma_{\beta_i}\to {\mathbb R}^4$, such that $F_i(x_i)=0$, $B_{r_0}(x_i)\cap\partial \Sigma_{\beta_i}=\emptyset$, $\cos\alpha_{i}\geq\delta$ and
\begin{equation}\label{E3.13}
    \int_{B_{r_0}(x_i)\cap\Sigma_i}|\textbf{A}_i|^2d\mu_i\to 0,
\end{equation}
but
\begin{equation}\label{E3.14}
    \max_{\sigma\in[0,r_0]}\sigma^2\sup_{B_{r_0-\sigma}(x_i)\cap \Sigma_{\beta_i}}|\textbf{A}_i|^2>4
\end{equation}
for all $i$.

Denote $r_i$ the extrinsic distance function on $\Sigma_{\beta_i}$ from $x_i$. Set $G_i=(r_0-r_i)^2|\textbf{A}_i|^2$ on $B_{r_0}(x_i)\cap\Sigma_{\beta_i}$. Let $y_i\in B_{r_0}(x_i)\cap\Sigma_{\beta_i}$ be the point where $G_i$ achieves its maximum. Then by our assumption (\ref{E3.14}), we see that
\begin{eqnarray}\label{E3.15}
  G_i(y_i)
   &=& (r_0-r_i(y_i))^2|\textbf{A}_i|^2(y_i)=\max_{B_{r_0}(x_i)\cap\Sigma_{\beta_i}}(r_0-r_i)^2|\textbf{A}_i|^2 \nonumber\\
   &\geq&   \max_{\sigma\in[0,r_0]}\sigma^2\sup_{B_{r_0-\sigma}(x_i)\cap \Sigma_{\beta_i}}|\textbf{A}_i|^2>4.
\end{eqnarray}
Choose $\sigma_i>0$ such that
\begin{equation*}
    \sigma_i^2|\textbf{A}_i|^2(y_i)=1.
\end{equation*}
Since $(r_0-r_i(y_i))^2|\textbf{A}_i|^2(y_i)>4$ by the choice of $y_i$, we see that
\begin{equation*}
    2\sigma_i<r_0-r_i(y_i).
\end{equation*}
Consequently, by the triangle inequality, we have on $B_{\sigma_i}(y_i)\cap\Sigma_{\beta_i}$ that
\begin{equation*}
    \frac{1}{2}\leq \frac{r_0-r_i}{r_0-r_i(y_i)}\leq 2.
\end{equation*}
Therefore, by the choice of $y_i$, we have
\begin{eqnarray*}
 (r_0-r_i(y_i))^2\sup_{B_{\sigma_i}(y_i)\cap\Sigma_{\beta_i}}|\textbf{A}_i|^2
   &\leq & 4\sup_{B_{\sigma_i}(y_i)\cap\Sigma_{\beta_i}}(r_0-r_i)^2|\textbf{A}_i|^2=4\sup_{B_{\sigma_i}(y_i)\cap\Sigma_{\beta_i}}G_i \\
   &\leq& 4G(y_i)=4(r_0-r_i(y_i))^2|\textbf{A}_i|^2(y_i).
\end{eqnarray*}
Dividing through by $(r_0-r_i(y_i))^2$ and using the definition of $\sigma_i$ gives
\begin{equation}\label{E3.16}
    \sup_{B_{\sigma_i}(y_i)\cap\Sigma_{\beta_i}}|\textbf{A}_i|^2\leq 4|\textbf{A}_i|^2(y_i)=4\sigma_i^{-2}.
\end{equation}
Now, we consider the rescaled immersion $\tilde F_i=\sigma_i^{-1}F_i:\tilde B_{1}(y_i)\cap\tilde \Sigma_{\beta_i}\to {\mathbb R}^4$ and its induced metric $\tilde{ds_i}^2=\sigma_i^{-2}ds_i^2$. Then we have $\partial\tilde B_{1}(y_i)\cap\tilde \Sigma_{\beta_i}=\emptyset$,
\begin{equation*}
   \sup_{\tilde B_{1}(y_i)\cap\tilde\Sigma_{\beta_i}}|\widetilde{\textbf{A}}_i|^2 =\sigma_i^{2}\sup_{B_{\sigma_i}(y_i)\cap\Sigma_{\beta_i}}|\textbf{A}_i|^2\leq 4,
\end{equation*}
and
\begin{equation*}
    |\widetilde{\textbf{A}}_i|^2(y_i)=\sigma_i^2|\textbf{A}_i|^2(y_i)=1.
\end{equation*}
Thus the sequence $\tilde F_i$ is a sequence of $\beta_i$-symplectic critical surfaces with uniformly bounded second fundamental form, translated so that $\tilde F_i(y_i)=0$. By taking a subsequence if necessary, we assume that $\beta_i\to\beta_0$. By the compactness theorem (Theorem \ref{thm-compact}), a subsequence converges smoothly on compact subsets, to a $\beta_0$-symplectic critical surface $\tilde F_{\infty}:\tilde B_{1}(y_{\infty})\cap\tilde \Sigma_{\beta_0}\to {\mathbb R}^4$. Furthermore, $\Sigma_{\beta_0}$ satisfies that: $ |\widetilde{\textbf{A}}_{\infty}|^2(y_{\infty})=1$, and by the scaling invariant property and (\ref{E3.13})
\begin{equation*}
     \int_{\tilde B_{1}(y_{\infty})\cap\tilde \Sigma_{\beta_0}}|\widetilde{\textbf{A}}_{\infty}|^2
     =\lim_{i\to\infty}\int_{\tilde B_{1}(y_i)\cap\tilde\Sigma_{\beta_i}}|\widetilde{\textbf{A}}_i|^2
     =\lim_{i\to\infty}\int_{B_{\sigma_i}(y_i)\cap\Sigma_i}|\textbf{A}_i|^2d\mu_i=0.
\end{equation*}
This gives the desired contradiction.
\hfill Q.E.D.

\vspace{.1in}

\section{Estimations of Lower Bound for $\cos\alpha$ and the Upper Bound for the Total Curvature}

In this section, we will give the lower bound of $\cos\alpha$ and the upper bound of the total curvature for a $\beta$-symplectic critical surface, which will be needed in the proof of the main theorem.

First, we will apply Moser iteration technique to provide lower bound of $\cos\alpha$ for $\beta$-symplectic critical surfaces. Our first main result in this section is as follows:

\begin{theorem}\label{THM5.1}
Let $\Sigma_{\beta}$ be a closed $\beta$-symplectic critical surface in a compact K\"ahler surface $M$, if
\begin{equation}\label{E5.15}
\int_{\Sigma_{\beta}}\frac{1}{\cos^q\alpha}d\mu\leq C_3<\infty
\end{equation}
for some constant $q>4$, then there exists a constant $\delta$ depending on $M$, $C_3$, $\beta$ and $q-4$, such that
\begin{equation}\label{E5.16}
\inf_{\Sigma_{\beta}}\cos\alpha\geq \delta>0.
\end{equation}
\end{theorem}

\vspace{.1in}
Recall that the Euler-Lagrange equation of the functional $L_\beta$
($\beta\neq -1$) is given by (\ref{me}), i.e.,
\begin{equation*}
\cos^3\alpha{\bf{H}}-\beta(J(J\nabla\cos\alpha)^\top)^\bot=0,
\end{equation*}
or equivalently (see Section 2 of \cite{HL}),
\begin{equation}\label{E5.2}
    \textbf{H}=\beta\frac{\sin^2\alpha}{\cos^2\alpha}\textbf{V},
\end{equation}
where in local orthonormal frame, $\textbf{V}=\nabla_{e_2}\alpha e_3+\nabla_{e_1}\alpha
e_4$. In particular, we have
\begin{equation}\label{E5.3}
    |\textbf{H}|=\beta\frac{\sin^2\alpha}{\cos^2\alpha}|\nabla\alpha|.
\end{equation}
First let's recall the elliptic equation satisfied by the K\"ahler angle on a $\beta$-symplectic critical surface:

\begin{theorem}\label{thm-equation}(\cite{HLS1})
Suppose that $M$ is K\"ahler surface and $\Sigma$ is a
$\beta$-symplectic critical surface in $M$ with K\"ahler angle
$\alpha$, then $\cos\alpha$ satisfies,
\begin{eqnarray}\label{E5.4}
\Delta\cos\alpha &=&\frac{2\beta\sin^2\alpha}{\cos\alpha
(\cos^2\alpha+\beta\sin^2\alpha)}|\nabla\alpha|^2-2\cos\alpha|\nabla\alpha|^2\nonumber\\
&&-\frac{\cos^2\alpha\sin^2\alpha}{\cos^2\alpha+\beta\sin^2\alpha}
Ric(Je_1, e_2).
\end{eqnarray}
\end{theorem}

\vspace{.1in}

In order to apply the Moser iteration technique, we also need the following Sobolev inequality for submanifolds:

\begin{theorem}\label{thm-sobolev}
(\textbf{Sobolev Inequality})(\cite{MS} Theorem 2.1) Let $\Sigma^n$ be a smooth immsered submanifold in $\textbf{R}^{n+k}$. Let $h$ be a
nonnegative, compactly supported Lipschitz continuous function on
$\Sigma$. Then the inequality
\begin{eqnarray}\label{E5.5}
    (\int_{\Sigma}h^{\frac{n}{n-1}}d\mu)^{\frac{n-1}{n}}
    \leq  c(n)\int_{\Sigma}(|\nabla h|+|\overline{\textbf H}|h) d\mu
\end{eqnarray}
holds. Here $\overline{\textbf H}$ is the mean curvature vector of $\Sigma^n$ in ${\mathbb R}^{n+k}$.
\end{theorem}

Actually, we will use the following variation of the above theorem:

\begin{theorem}\label{thm-sobolev2}
(\textbf{Sobolev Inequality}) Under the assumptions of Theorem \ref{thm-sobolev}, we have
\begin{equation}\label{E5.6}
\left(\int_{\Sigma}h^{\frac{2n}{n-1}}d\mu\right)^{\frac{n-1}{n}} \leq
      c(n)\int_{\Sigma}[|\nabla h|^2+h^2+h^2|\overline{\textbf H}|]d\mu.
\end{equation}
\end{theorem}

\begin{proof}
Replacing $h$ by $h^2$ and applying Cauchy inequality in (\ref{E5.5}) yields (\ref{E5.6}).
\end{proof}

\vspace{.1in}

Now we start to proceed the Moser iteration procedure to prove Theorem \ref{THM5.1}.

\vspace{.1in}

\textbf{ Proof of Theorem \ref{THM5.1}:} In the following we will denote $\Sigma_{\beta}$ by $\Sigma$. Set $f=\frac{1}{\cos\alpha}$ , then
\begin{equation}\label{E5.7}
\nabla f=-\frac{\nabla\cos\alpha}{\cos^{2}\alpha} = \frac{\sin\alpha\nabla\alpha}{\cos^{2}\alpha}=\frac{\sin\alpha}{\cos\alpha}f\nabla\alpha.
\end{equation}
By (\ref{E5.3}) and (\ref{E5.7}), we have
\begin{equation}\label{E5.8}
|{\bf H}|=\beta\frac{\sin\alpha}{\cos\alpha}f^{-1}|\nabla f|.
\end{equation}
Furthermore, from (\ref{E5.4}), we have
\begin{eqnarray}\label{E-alpha}
\Delta f &=& \Delta\frac{1}{\cos^{}\alpha}\nonumber\\
&=&\frac{2}{\cos\alpha(\cos^2\alpha+\beta\sin^2\alpha)}|\nabla\alpha|^2
+\frac{\cos^2\alpha\sin^2\alpha}{\cos^{2}\alpha(\cos^2\alpha+\beta\sin^2\alpha)} Ric(Je_1, e_2).
\end{eqnarray}
Since $|Ric(Je_1,e_2)|\leq K_0$, we have
\begin{eqnarray}\label{E5.9}
\Delta f &\geq& -\frac{\cos^2\alpha\sin^2\alpha}{\cos^{2}\alpha(\cos^2\alpha+\beta\sin^2\alpha)}K_0=-\frac{\sin^2\alpha}{\cos\alpha}\frac{\cos^2\alpha}{\cos^2\alpha+\beta\sin^2\alpha}K_0f\nonumber\\
& \geq & -\frac{\sin^2\alpha}{\cos\alpha}K_0 f.
\end{eqnarray}
Multiplying both side of (\ref{E5.9}) by $\eta^2f^{p-1}$ and integrating by parts with cutoff function $\eta$ to be determined later yield
\begin{equation*}
K_0\int_{\Sigma}\frac{\sin^2\alpha}{\cos\alpha}\eta^2f^{p}d\mu \geq-\int_{\Sigma}\eta^2f^{p-1}\Delta fd\mu
=2\int_{\Sigma}\eta f^{p-1}\nabla\eta\nabla fd\mu+(p-1)\int_{\Sigma}\eta^2f^{p-2}|\nabla f|^2d\mu,
\end{equation*}
which implies that
\begin{eqnarray*}
(p-1)\int_{\Sigma}\eta^2f^{p-2}|\nabla f|^2d\mu
&\leq&2\int_{\Sigma}\eta f^{p-1}|\nabla\eta|\cdot|\nabla f|d\mu+ K_0\int_{\Sigma}\frac{\sin^2\alpha}{\cos\alpha}\eta^2f^{p}d\mu\\
&\leq& \frac{p-1}{2}\int_{\Sigma}\eta^2f^{p-2}|\nabla f|^2d\mu+\frac{2}{p-1}\int_{\Sigma}f^{p}|\nabla\eta|^2d\mu\\
& & +K_0\int_{\Sigma}\frac{\sin^2\alpha}{\cos\alpha}\eta^2f^{p}d\mu.
\end{eqnarray*}
Therefore, we have
\begin{equation*}
(p-1)^2\int_{\Sigma}\eta^2f^{p-2}|\nabla f|^2d\mu\leq4\int_{\Sigma}f^{p}|\nabla\eta|^2d\mu+2(p-1) K_0\int_{\Sigma}\frac{\sin^2\alpha}{\cos\alpha}\eta^2f^{p}d\mu.
\end{equation*}
For $p\geq 2$, we have $(p-1)^2\geq \frac{p^2}{4}$, thus we have
\begin{equation}\label{E5.10}
\frac{p^2}{4}\int_{\Sigma}\eta^2f^{p-2}|\nabla f|^2d\mu\leq4\int_{\Sigma}f^{p}|\nabla\eta|^2d\mu+2(p-1) K_0\int_{\Sigma}\frac{\sin^2\alpha}{\cos\alpha}\eta^2f^{p}d\mu,
\end{equation}
which can be rewritten as
\begin{equation*}
\int_{\Sigma}\eta^2|\nabla f^{\frac{p}{2}}|^2d\mu\leq4\int_{\Sigma}f^{p}|\nabla\eta|^2d\mu+2(p-1) K_0\int_{\Sigma}\frac{\sin^2\alpha}{\cos\alpha}\eta^2f^{p}d\mu.
\end{equation*}
Since
\begin{equation*}
|\nabla (\eta f^{\frac{p}{2}})|^2\leq2\eta^2|\nabla f^{\frac{p}{2}}|^2+2f^p|\nabla \eta|^2,
\end{equation*}
we have
\begin{equation*}
\int_{\Sigma}|\nabla (\eta f^{\frac{p}{2}})|^2d\mu\leq10\int_{\Sigma}f^{p}|\nabla\eta|^2d\mu+4(p-1) K_0\int_{\Sigma}\frac{\sin^2\alpha}{\cos\alpha}\eta^2f^{p}d\mu.
\end{equation*}

\vspace{.1in}

Next, we will apply for the Sobolev inequality, i.e., Theorem \ref{thm-sobolev}. For this purpose, we first note that by Nash embedding theorem, we can embed $(M^4,\bar g)$ into some Euclidean space ${\mathbb R}^N$ isometrically for some $N$. For $\Sigma^2\subset M^4\subset{\mathbb R}^N$, if we denote the mean curvature vectors of $\Sigma$ in $M$ and ${\mathbb R}^N$ by $\textbf H$ and $\overline{\textbf H}$, respectively, then an easy computation shows that
\begin{equation*}
\overline{\textbf H}=\textbf H+\textbf E,
\end{equation*}
where $\textbf E$ depends only on the second fundamental form of $M$ in ${\mathbb R}^N$. In particular, $|\textbf E|\leq C$ for some constant $C$.

Now we can apply the Sobolev inequality (\ref{E5.6}) with $h=\eta f^{\frac{p}{2}}$ and $n=2$ to obtain
\begin{eqnarray*}
\left(\int_{\Sigma}\eta^4 f^{2p}d\mu\right)^{\frac{1}{2}}
&\leq&  C\int_{\Sigma}\left(|\nabla (\eta f^{\frac{p}{2}})|^2+\eta^2f^{p}+\eta^2f^{p}|{\bf H}|\right)d\mu\\
&\leq& C\int_{\Sigma}f^{p}|\nabla\eta|^2d\mu+C(p-1) K_0\int_{\Sigma}\frac{\sin^2\alpha}{\cos\alpha}\eta^2f^{p}d\mu\\
& & +C\int_{\Sigma}\eta^2f^{p}d\mu+C\int_{\Sigma}\eta^2f^{p}|{\bf H}|d\mu\\
&\leq& C\int_{\Sigma}f^{p}|\nabla\eta|^2d\mu+Cp\int_{\Sigma}\frac{\eta^2f^{p}}{\cos\alpha}d\mu+C\int_{\Sigma}\eta^2f^{p}|{\bf H}|d\mu.
\end{eqnarray*}
Plugging (\ref{E5.8}) into the above inequality and using (\ref{E5.10}), we can get that
\begin{eqnarray}\label{E5.11}
\left(\int_{\Sigma}\eta^4 f^{2p}d\mu\right)^{\frac{1}{2}}
&\leq& C\int_{\Sigma}f^{p}|\nabla\eta|^2d\mu+Cp\int_{\Sigma}\frac{\eta^2f^{p}}{\cos\alpha}d\mu+C(\beta)\int_{\Sigma}\frac{\sin\alpha}{\cos\alpha}\eta^2f^{p-1}|\nabla f|d\mu\nonumber\\
&\leq& C\int_{\Sigma}f^{p}|\nabla\eta|^2d\mu+Cp\int_{\Sigma}\frac{\eta^2f^{p}}{\cos\alpha}d\mu+\frac{p^2}{4}\int_{\Sigma}\eta^2f^{p-2}|\nabla f|^2d\mu\nonumber\\
& & +\frac{C}{p^2}\int_{\Sigma}\frac{\sin^2\alpha}{\cos^2\alpha}\eta^2f^{p}d\mu\nonumber\\
&\leq& C\int_{\Sigma}f^{p}|\nabla\eta|^2d\mu+Cp\int_{\Sigma}\frac{\eta^2f^{p}}{\cos^2\alpha}d\mu.
\end{eqnarray}
Now we fix $R$ with $0<R<1$ to be determined later. Define $R_k=R(1+\frac{1}{2^k})$, then $R_0=2R$ and $R_{\infty}=R$. Choose $\eta_k\in C^{1}_c(B_{R_k}\cap\Sigma)$, $0\leq\eta_k\leq 1$, $\eta_k\equiv1$ on $B_{R_{k+1}}\cap\Sigma$, and
\begin{equation*}
|\nabla\eta_k|\leq\frac{2}{R_k-R_{k+1}}=\frac{2^{k+2}}{R}.
\end{equation*}
Then we have from (\ref{E5.11}) that
\begin{eqnarray*}
\left(\int_{B_{R_{k+1}}\cap\Sigma} f^{2p}d\mu\right)^{\frac{1}{2}}
&\leq& \frac{C4^k}{R^2}\int_{B_{R_{k}}\cap\Sigma}f^{p}d\mu+Cp\int_{B_{R_{k}}\cap\Sigma}\frac{f^{p}}{\cos^2\alpha}d\mu\\
&\leq& \frac{C4^kp}{R^2}\int_{B_{R_{k}}\cap\Sigma}\frac{f^{p}}{\cos^2\alpha}d\mu
=  \frac{C4^kp}{R^2}\int_{B_{R_{k}}\cap\Sigma}f^{p+2}d\mu
\end{eqnarray*}
that is,
\begin{equation*}
\left(\int_{B_{R_{k+1}}\cap\Sigma} f^{2p}d\mu\right)^{\frac{1}{2p}}\leq C^{\frac{1}{p}}4^{\frac{k}{p}}p^{\frac{1}{p}}R^{-\frac{2}{p}}\left(\int_{B_{R_{k}}\cap\Sigma}f^{p+2}d\mu\right)^{\frac{1}{p}}.
\end{equation*}
If $p>2$, then $2p>p+2$, and we have
\begin{equation}\label{E5.12}
||f||_{L^{2p}(B_{R_{k+1}}\cap\Sigma)} \leq C^{\frac{1}{p}}4^{\frac{k}{p}}p^{\frac{1}{p}}R^{-\frac{2}{p}}||f||_{L^{p+2}(B_{R_{k}}\cap\Sigma)}^{1+\frac{2}{p}}.
\end{equation}
Now we set $p_k=2^k\left(p_0-4\right)+4=4(2^{k}a+1)$ (with $p_0=4(1+a)>4$) and $p=2(2^{k+1}a+1)$, then we have
\begin{equation*}
2p=p_{k+1}, \ \ p+2=p_k.
\end{equation*}
By (\ref{E5.12}), we get that
\begin{eqnarray*}
&  &  ||f||_{L^{p_{k+1}}(B_{R_{k+1}}\cap\Sigma)}\\
&\leq& C^{\frac{1}{2}\cdot\frac{1}{2^{k+1}a+1}}R^{-\frac{1}{2^{k+1}a+1}}
   2^{\frac{1}{2}\cdot\frac{1}{2^{k+1}a+1}}\left(2^{k+1}a+1\right)^{\frac{1}{2}\cdot\frac{1}{2^{k+1}a+1}}4^{\frac{1}{2}\cdot\frac{k}{2^{k+1}a+1}}||f||_{L^{p_k}(B_{R_{k}}\cap\Sigma)}^{1+\frac{1}{2^{k+1}a+1}}\\
&\leq & C^{\frac{1}{2}\sum_{l=0}^{k}\frac{1}{2^{l+1}a+1}\left(1+\frac{1}{2^{l+2}a+1}\right)\cdots\left(1+\frac{1}{2^{k+1}a+1}\right)}4^{\frac{1}{2}\sum_{l=1}^{k}\frac{l-1}{2^{l+1}a+1}\left(1+\frac{1}{2^{l+2}a+1}\right)\cdots\left(1+\frac{1}{2^{k+1}a+1}\right)} \\
& & \cdot2^{\frac{1}{2}\sum_{l=0}^{k}\frac{1}{2^{l+1}a+1}\left(1+\frac{1}{2^{l+2}a+1}\right)\cdots\left(1+\frac{1}{2^{k+2}a+1}\right)}R^{-\sum_{l=0}^{k}\frac{1}{2^{l+1}a+1}\left(1+\frac{1}{2^{l+2}a+1}\right)\cdots\left(1+\frac{1}{2^{k+1}a+1}\right)}\\
& &\cdot\Pi_{l=0}^{k+1}\left(1+2^{l+1}a\right)^{\frac{1}{2}\frac{1}{1+2^{l+1}a}\left(1+\frac{1}{2^{l+2}a+1}\right)\cdots\left(1+\frac{1}{2^{k+1}a+1}\right)}||f||_{L^{p_0}(B_{2R}\cap\Sigma)}^{\left(1+\frac{1}{2a+1}\right)\cdots\left(1+\frac{1}{2^{k+1}a+1}\right)}.
\end{eqnarray*}
It is easy to see that
\begin{equation*}
\Pi_{k=0}^{\infty}\left(1+\frac{1}{2^{k+1}a+1}\right)\equiv A<\infty.
\end{equation*}
Therefore, we get that
\begin{eqnarray*}
 ||f||_{L^{p_{k+1}}(B_{R_{k+1}}\cap\Sigma)}
&\leq & C^{\frac{A}{2}\sum_{l=0}^{\infty}\frac{1}{2^{l+1}a+1}}4^{\frac{A}{2}\sum_{l=1}^{\infty}\frac{l-1}{2^{l+1}a+1}}2^{\frac{A}{2}\sum_{l=0}^{\infty}\frac{1}{2^{l+1}a+1}}R^{-A\sum_{l=0}^{k}\frac{1}{2^{l+1}a+1}}\\
& &\cdot\Pi_{l=0}^{k+1}\left(1+2^{l+1}a\right)^{\frac{A}{2}\frac{1}{1+2^{l+1}a}}||f||_{L^{p_0}(B_{2R}\cap\Sigma)}^{A}\\
&\leq & C(M,\beta,a)R^{-A\sum_{l=0}^{k+1}\frac{1}{2^{l+1}a}}2^{\frac{A}{2}\sum_{l=0}^{\infty}\frac{l+1}{1+2^{l+1}a}}||f||_{L^{p_0}(B_{2R}\cap\Sigma)}^{A}\\
&\leq & C(M,\beta,a)R^{-\frac{A}{a}}||f||_{L^{p_0}(B_{2R}\cap\Sigma)}^{A}.
\end{eqnarray*}
Letting $k\to\infty$, we finally obtain that
\begin{equation}\label{E5.13}
\sup_{B_{R}\cap\Sigma}f \leq C(M,\beta,a)R^{-\frac{A}{a}}||f||_{L^{p_0}(B_{2R}\cap\Sigma)}^{A}.
\end{equation}
If we take $R=\min\{i_0(M),1\}$, where $i_0(M)$ is the injectivity radius of $M$, then we finally get that
\begin{equation}\label{E5.14}
\sup_{\Sigma}\frac{1}{\cos\alpha} \leq C(M,\beta,a)\left(\int_{\Sigma}\frac{1}{\cos^{p_0}\alpha}\right)^{\frac{A}{p_0}}.
\end{equation}
Now we choose $p_0=q>4$ so that $a=\frac{q-4}{4}$, then the conclusion of the theorem follows.
\hfill Q.E.D.

\vspace{.1in}

Next we will control the total curvature of $\beta$-symplectic critical surfaces.

\vspace{.1in}
\begin{theorem}\label{THM5.5}
Let $\Sigma_{\beta}$ be a closed $\beta$-symplectic critical surface in a compact K\"ahler surface $M$, if
\begin{equation}\label{E0.15}
\int_{\Sigma_{\beta}}\frac{1}{\cos^q\alpha}d\mu\leq C_3<\infty
\end{equation}
for some constant $q>4$ and the genus $g(\Sigma_{\beta})$ of the $\Sigma_{\beta}$ is bounded from above by $g_0$, then there exists a constant $C_2$ depending on $M$, $C_3$, $\beta$, $g_0$ and $q-4$, such that
\begin{equation}\label{E0.16}
\int_{\Sigma_{\beta}}|{\bf A}|^2\leq C_2.
\end{equation}
\end{theorem}

{\it Proof.} By Theorem \ref{THM5.1}, we know that
$$\cos\alpha\geq\delta>0.$$
Multiplying the inequality (\ref{E5.9}) by $f=\frac{1}{\cos\alpha}$ and integrating by parts, we get that
\begin{equation*}
\int_{\Sigma_\beta}|\nabla f|^2\leq \frac{K_0}{\delta}\int_{\Sigma_\beta}f^2\leq \frac{K_0}{\delta^3}Area(\Sigma_{\beta})
\leq \frac{K_0}{\delta^3}C_3,
\end{equation*} which implies that
\begin{equation}\label{E0.19}
\int_{\Sigma_\beta}|\nabla \cos\alpha|^2\leq \int_{\Sigma_\beta}|\nabla f|^2\leq \frac{K_0}{\delta^3}C_3.
\end{equation}
By (\ref{E5.3}), the lower bound of $\cos\alpha$ and (\ref{E0.19}), we get that
\begin{equation}\label{E0.20}
\int_{\Sigma_{\beta}}|{\bf H}|^2
\leq  \beta^2 \int_{\Sigma_{\beta}}\frac{\sin^2\alpha}{\cos^4\alpha}|\nabla\cos\alpha|^2 
\leq \frac{\beta^2}{\delta^4} \int_{\Sigma_{\beta}}|\nabla\cos\alpha|^2 \leq  \frac{\beta^2}{\delta^7}K_0C_3.
\end{equation}
On the other hand,  by Gauss equation, we have:
\begin{eqnarray*}
    K_{\Sigma}=R_{1212}=K_{1212}+\frac{1}{2}(|\textbf{H}|^2-|\textbf{A}|^2).
\end{eqnarray*}
where $K$ is the curvature tensor of $M$ and $R$ is the curvature tensor of $\Sigma_{\beta}$ with induced metric. Suppose the sectional curvature of $M$ is bounded from above by $K_0$, i.e., $|K_M|\leq K_0$, then for any surface $\Sigma$ in $M$, we have
\begin{eqnarray*}
    |\textbf{A}|^2\leq 2K_0-2K_{\Sigma}+|\textbf{H}|^2.
\end{eqnarray*}
Integrating the above inequality over $\Sigma_{\beta}$, using Gauss-Bonnet Theorem, we finally get that
\begin{eqnarray}\label{E0.21}
\int_{\Sigma_{\beta}}|{\bf A}|^2
&\leq & 2K_0Area(\Sigma_{\beta})-4\pi\chi(\Sigma_{\beta})+\int_{\Sigma_{\beta}}|\textbf{H}|^2d\mu_{\beta} \nonumber\\
&= & 2K_0Area(\Sigma_{\beta})+8\pi (g(\Sigma_{\beta})-1)+\int_{\Sigma_{\beta}}|\textbf{H}|^2d\mu_{\beta}.
\end{eqnarray}
Combining this with (\ref{E0.20}) and the fact that $Area(\Sigma_{\beta})\leq C_3$ proves the theorem.

\hfill Q.E.D.

\vspace{.1in}

\section{Compactness of $\beta$-Symplectic Critical Surface}

In this section, we will prove the main result in this paper.

\begin{theorem}\label{THM2.1}
Let $M$ be a closed K\"ahler surface and $\Sigma_i\subset M$ a sequence of closed $\beta_i$-symplectic critical surfaces with $\beta_i\to\beta_0\in (0,\infty)$,

\begin{equation}\label{E3.3}
   \int_{\Sigma_i}\frac{1}{\cos^q\alpha_i}\leq C_1,
\end{equation}
for some $q>4$ and $g(\Sigma_{\beta_i})\leq g_0$, where $g(\Sigma_{\beta_i})$ is the genus of $\Sigma_{\beta_i}$. Then there exists a finite set of points ${\mathcal S}\subset M$ and a subsequence $\Sigma_{i'}$ that converges uniformly in the $C^l$ topology (for any $l<\infty$) on compact subsets of $M\backslash {\mathcal S}$ to a $\beta_0$-symplectic critical surface $\Sigma\subset M$. The subsequence also converges to $\Sigma$ in (extrinsic) Hausdorff distance.

Furthermore, around each singular point $p_{\gamma}\in {\mathcal S}$, there is a bubble which is a smooth holomorphic curve in ${\mathbb C}^2$. The tangent cone of $\Sigma\cup{\mathcal S}$ at $p_{\gamma}$ is a flat cone consists of union of planes in ${\mathbb C}^2$ which intersects at only one point.

Moreover, for each connected component $\Omega$ of $\Sigma$ with $p_{\gamma}\in{\mathcal S}\cap \bar\Omega$, $\Omega\cup \{p_{\gamma}\}$ is a smooth $\beta_0$-symplectic critical surface of $M$.
\end{theorem}

\vspace{.1in}

\textbf{Proof}: Within the proof, $\varepsilon=\varepsilon(M)>0$  and $r_0=r_0(M)>0$ will be from Theorem \ref{thm-small}. First note that by the assumption (\ref{E3.3}) and Theorem \ref{THM5.1}, we see that there exists a positive constant $\delta$ depending on $M$, $\beta_0$, $q-4$ and $C_1$, such that
\begin{equation}\label{E4.2}
\cos\alpha\geq\delta>0.
\end{equation}
By Theorem \ref{THM5.5}, there exists a constant $C_2$ depending on$M$, $\beta_0$, $q-4$, $g_0$ and $C_1$ such that
\begin{equation}\label{E4.3}
\int_{\Sigma_{\beta_i}}|{\textbf{A}_i}|^2\leq C_2.
\end{equation}

In order to find the set ${\mathcal S}$, we define measure $\nu_i$ by
\begin{equation}\label{E3.17}
  \nu_i(U)=\int_{U\cap\Sigma_i} |{\textbf A}_i|^2\leq C_2,
\end{equation}
where $U\subset M$ and ${\textbf{A}_i}={\textbf{A}_{\Sigma_i}}$. The general compactness theorem for Radon measure implies that there is a subsequence $\nu_{k_i}$ which converges weakly to a Radon measure $\nu$ with
\begin{equation}\label{E3.18}
  \nu(M)\leq C_2.
\end{equation}
For ease of notation, replace $\nu_{k_i}$ by $\nu_i$. We define the set
$${\mathcal S}=\{x\in M\mid \nu(x)\geq \varepsilon\}.$$
Then ${\mathcal S}$ consists of at most $\varepsilon^{-1}C_2$ points. Suppose ${\mathcal S}=\{p_1,\cdots,p_l\}$.

Given any $y\in M\backslash{\mathcal S}$, we have $\nu(y)<\varepsilon$. Since $\nu$ is a Radon measure, and hence Borel regular, there exists some $0\leq 10s <\min\{r_0,i_0\}\equiv r_1$ such that
\begin{equation}\label{E3.19}
  \nu(B_{10s}(y))<\varepsilon.
\end{equation}
Since $\nu_i\to\nu$, (\ref{E3.19}) implies that for $i$ sufficiently large,
\begin{equation}\label{E3.20}
  \int_{B_{10s}(y)\cap\Sigma_i} |{\textbf A}_i|^2=\nu_i(B_{10s}(y))<\varepsilon.
\end{equation}
Then by Small Regularity Theorem (Theorem \ref{thm-small}), for sufficiently large $i$ and $z\in B_{5s}(y)\cap\Sigma_i$,
\begin{equation}\label{E3.21}
25s^2|{\textbf A}_i|^2(z)\leq1.
\end{equation}

Next, we show that the number of connected components of $B_s(y)\cap\Sigma_i$ which intersects $B_{\frac{s}{2}}(y)$ is uniformly bounded independent of both $y$ and $i$. Actually, take any such connected component $\Gamma$ and choose one point $z\in B_{\frac{s}{2}}(y)\cap \Gamma$. Since $\cos\alpha\geq\delta>0$ on $\Gamma$, by Proposition 2.1 of \cite{HL3}, we can easily see that
\begin{equation*}
Area(\hat{B}_r(z))\geq C_0r^2,
\end{equation*}
where $C_0$ is a constant depending only on $\delta$, but not depending on $i$ and $z$. Here, $\hat{B}_r(z)$ denotes the intrinsic ball of radius $r$ around $z$. Since $\hat{B}_r(z)\subset B_r(z)\cap\Sigma_i$, we get that
\begin{equation}\label{E3.22}
Area(B_s(y)\cap\Gamma)\geq Area(B_{\frac{s}{2}}(z)\cap\Gamma)\geq Area(\hat{B}_{\frac{s}{2}}(z))\geq C_0(\frac{s}{2})^2=\frac{C_0}{4}s^2.
\end{equation}
Let $c_{y,i}$ denote the number of connected components of $B_s(y)\cap\Sigma_i$ which intersects $B_{\frac{s}{2}}(y)$. By (\ref{E3.22}), we have
\begin{equation*}
Area(B_s(y)\cap\Sigma_i)\geq \frac{1}{4}C_0s^2 c_{y,i}.
\end{equation*}
Combining with (\ref{E3.6}) and (\ref{E4.3}), we have
\begin{eqnarray*}
 \frac{1}{4}C_0 c_{y,i}
 & \leq & \frac{Area(B_s(y)\cap\Sigma_i)}{s^2}\\
   &\leq & 2e^{4\sqrt{K_0}r_1}\left(\frac{Area(B_{r_1}(y)\cap\Sigma_i)}{r_1^2}+\int_{B_{r_1}(y)\cap\Sigma_i}|\textbf{H}_i|^2\right).\\
   &\leq& 2e^{4\sqrt{K_0}i_0}\left(\frac{C_1}{r_1^2}+2C_2\right),
\end{eqnarray*}
which implies that
\begin{equation*}
c_{y,i}\leq \frac{8e^{4\sqrt{K_0}i_0}}{C_0}\left(\frac{C_1}{r_1^2}+2C_2\right).
\end{equation*}
In particular, the number of connected components of $B_s(y)\cap\Sigma_i$ which intersects $B_{\frac{s}{2}}(y)$ is uniformly bounded independent of both $y$ and $i$.

Since we have a uniform estimate on the number of components, by Theorem \ref{thm-compact}, we see that there is a subsequence $\eta_i$, which converges in $B_s(y)$ smoothly. Since we can cover $M\backslash {\mathcal S}$ by countably many balls like this a diagonal argument finishes off the convergence to a limit surface $\Sigma$, which is smooth in $M\backslash {\mathcal S}$. This implies that $\Sigma$ satisfies (\ref{E4.3}). Furthermore, the uniform estimate implies that $\Sigma$ is a $\beta_0$-symplectic critical surface in $M$. We may also suppose that $\Sigma_i$ converges to $\Sigma$ in the sense of varifold.

We now show that $\Sigma_i$ converges to $\Sigma$ in Hausdorff distance. To see this, note that: if $\Sigma_i$ does not converges to $\Sigma$ in Hausdorff distance, then there exists a subsequence $\Sigma_{k_i}$ and points $y_{k_i}\in\Sigma_{k_i}$, such that $dist(y_{k_i},\Sigma)>2\xi>0$ for some $\xi>0$. On the other hand, by (\ref{E3.22}), we get that $Area(B_{\xi}(y_{k_i})\cap \Sigma_{k_i})\geq\frac{C_0}{4}\xi^2$. Since varifold convergence implies area measure converges, for $i$ sufficiently, we must have $Area(B_{\xi}(y_{k_i})\cap \Sigma)>0$, which contradicts with the fact that $B_{\xi}(y_{k_i})\cap \Sigma=\emptyset$. This shows that $\Sigma_i$ converges to $\Sigma$ in Hausdorff distance.

\vspace{.1in}

Next, we analyze the bubble around each singular point. Since ${\mathcal S}$ consists of $l$ points, we can choose $\eta>0$, such that each pair of $B_{\eta}(p_{\beta})$ and $B_{\eta}(p_{\gamma})$ are disjoint for $1\leq \beta,\gamma\leq l$. Now fix $p_{\gamma}\in{\mathcal S}$. Set
\begin{equation*}
\lambda_i=\max_{\Sigma_i\cap B_{\eta}(p_{\gamma})}|\textbf{A}_i|=|\textbf{A}_i(x_i)|,
\end{equation*}
then we know that $\lambda_i\to\infty$ and $x_i\to p_{\gamma}$. When $\eta$ is sufficiently small, we take the normal coordinate on $B_{\eta}(p_{\gamma})\subset M$ around $p_{\gamma}$, and consider the rescaled surface
\begin{equation*}
\tilde\Sigma_i=\lambda_i(\Sigma_i-x_i).
\end{equation*}
Denote $\tilde{\textbf A}_i$ the second fundamental form on $\tilde\Sigma_i$. Then we see that
\begin{equation*}
|\tilde{\textbf A}_i|\leq 1,  \ in \ B_{\lambda_i\eta}(0);   \ \ \ \ |\tilde{\textbf A}_i|(0)=1.
\end{equation*}
Therefore, by Theorem \ref{thm-compact}, there exists a subsequence of $\tilde\Sigma_i$, which converges to a complete $\beta_0$-symplectic critical surface $\tilde{\Sigma}_{\infty}$ smoothly locally in ${\mathbb C}^2$.  $\tilde{\Sigma}_{\infty}$ is called the $bubble$ at $p_{\gamma}$. Since K\"ahler angle is scaling invariant, we know that $\tilde{\Sigma}_{\infty}$ is a complete $\beta_0$-symplectic critical surface in ${\mathbb C}^2$ with $\cos\tilde\alpha_{\infty}\geq\delta>0$.

Now we show that $\tilde\Sigma_{\infty}$ has quadratic extrinsic area growth. Namely, there exists a constant $B$, such that  for each $R>0$, we have
\begin{equation}\label{E3.23}
Area(B_R(0)\cap\tilde\Sigma_{\infty})\leq BR^2.
\end{equation}
Actually, from the monotonicity formula (\ref{E3.6}), we know that for each $x_0\in M$ and $0<s<i_0=i_0(M)$,
\begin{eqnarray}\label{area}
s^{-2}Area(B_{s}(x_0)\cap\Sigma_i)
   &\leq & e^{4\sqrt{K_0}i_0}\left(\frac{Area(B_{i_0}(x_0)\cap\Sigma)}{i_0^2}+\int_{\Sigma_i}|\textbf{H}|^2d\mu_i\right)\nonumber\\
   &\leq& e^{4\sqrt{K_0}i_0}\left(\frac{C_1}{i_0^2}+2C_2\right)\equiv B
\end{eqnarray}
By scaling property, for each fixed $R>0$
\begin{equation*}
R^{-2}Area(B_{R}(0)\cap\tilde\Sigma_i)=(\lambda_i^{-1}R)^{-2}Area(B_{\lambda_i^{-1}R}(x_i)\cap\Sigma_i)\leq B,
\end{equation*}
for sufficiently large $i$. Letting $i\to\infty$, we get (\ref{E3.23}).
Since intrinsic balls are contained in extrinsic balls, $Area(B^{\tilde\Sigma_\infty}_s(0)\cap\tilde\Sigma_{\infty})\leq Bs^2$.
By (\ref{E-alpha}), we have on $\Sigma_{\infty}$
\begin{equation*}
\Delta \frac{1}{\cos\alpha}=\frac{2}{\cos\alpha
(\cos^2\alpha+\beta_0\sin^2\alpha)}|\nabla\alpha|^2.
\end{equation*}
Since $\cos\alpha\geq\delta>0$, we have $1\leq\frac{1}{\cos\alpha}\leq\frac{1}{\delta}$, which means that $\frac{1}{\cos\alpha}$ is a subharmonic function bounded from above on $\tilde\Sigma_\infty$. However, the quadratic area growth implies that $\tilde\Sigma_\infty$ is parabolic (\cite{CY}). This forces $\frac{1}{\cos\alpha}$ to be a constant on $\tilde\Sigma_\infty$, which implies that $\tilde\Sigma_\infty$ is holomorphic with respect to some compatible complex structure in ${\mathbb C}^2$.

\vspace{.1in}

Now we will consider the tangent cone of $\Sigma\cup{\mathcal S}$ at $p_{\gamma}\in{\mathcal S}$. For this purpose, again, we choose the normal coordinate on $B_{\eta}(p_{\gamma})\subset M$ around $p_{\gamma}$, and consider the rescaled surface
\begin{equation*}
\hat\Sigma_{\lambda}=\lambda(\Sigma-p_{\gamma}), \ \ \ \ \ \lambda\in(1,\infty).
\end{equation*}
Since $\Sigma$ is smooth outside $M\backslash{\mathcal S}$, we see that $\hat\Sigma_{\lambda}$ is a $\beta_0$-symplectic critical surface for each $\lambda$. We show that there exists a subsequence $\hat\Sigma_{\lambda_i}$ converging to a varifold $\hat\Sigma_{\infty}$ in the sense of varifold. Actually, since each $\Sigma_i$ is a smooth surface satisfying (\ref{E4.3}) and (\ref{E3.3}), applying the monotonicity formula (\ref{E3.6}) with $s_2=\min\{i_0,\frac{1}{\sqrt{K_0}}\}$, we see that there exists a constant $C$, depending on $M$, $C_1$ and $C_2$, such that for any $0<r<\min\{i_0,\sqrt{K_0}\}$,
\begin{equation*}
\frac{Area(B_{r}(p_{\gamma})\cap\Sigma_i)}{r^2}\leq C.
\end{equation*}
Since $\Sigma_i$ converges to $\Sigma$ in the sense of varifold, we see that for any $0<r<\min\{i_0.\sqrt{K_0}\}$,
\begin{equation*}
\frac{Area(B_{r}(p_{\gamma})\cap\Sigma)}{r^2}\leq C.
\end{equation*}
Now, for any fixed $R>0$,
\begin{equation*}
\frac{Area(B_{R}(0)\cap\hat\Sigma_{\lambda})}{R^2}=\frac{Area(B_{\lambda^{-1}R}(p_{\gamma})\cap\Sigma)}{(\lambda^{-1}R)^2}\leq C,
\end{equation*}
for sufficiently large $\lambda$. Also we have
\begin{equation*}
\int_{\hat\Sigma_{\lambda}}|\hat{\textbf H}_{\lambda}|^2d\hat\mu_{\lambda}=\int_{\Sigma}|{\textbf H}|^2d\mu\leq 2C_2.
\end{equation*}
Then Allard Compactness Theorem (Theorem 42.7 of \cite{Simon2}) implies that $\hat\Sigma_{\lambda}$ converges to some $\hat\Sigma_{\infty}$ in the sense of varifold. $\hat\Sigma_{\infty}$ is called the $tangent\ cone$ of $\Sigma$ at $p_{\gamma}$.

Next we show that $\hat\Sigma_{\infty}$ consists of a finite union of plane intersecting at the origin. For this purpose, we need to show that $\hat\Sigma_{\infty}$ has at most one singular point. Now, we take any $x_0\in \hat\Sigma_{\infty}\backslash\{0\}$. Fix $R>0$ such that $0\notin B_{2R}(x_0)\cap\hat\Sigma_{\infty}$. Since each $B_{R}(x_0)\cap\hat\Sigma_{\lambda}$ is a smooth $\beta_0$-symplectic critical surface, the Small Regularity Theorem (Theorem \ref{thm-small}) implies that we have
\begin{equation*}
R^2\sup_{B_{\frac{R}{2}}(x_0)\cap\hat\Sigma_{\lambda}}|\hat{\textbf A}_{\lambda}|^2\leq C\int_{B_{R}(x_0)\cap\hat\Sigma_{\lambda}}|\hat{\textbf A}_{\lambda}|^2d\mu_{\lambda}=C\int_{B_{\lambda^{-1}R}(p_{\gamma}+\lambda^{-1}x_0)\cap\Sigma}|\textbf A|^2d\mu
\end{equation*}
It is obvious to see that for each fixed constant $\xi>0$, we have $B_{\lambda^{-1}R}(p_{\gamma}+\lambda^{-1}x_0)\cap\Sigma\subset B_{\xi}(p_{\gamma})\cap\Sigma$ for sufficiently large $\lambda$. By absolute continuity of integral, we can see that $B_{\frac{R}{2}}(x_0)\cap\hat\Sigma_{\lambda}$ converges smoothly to $B_{\frac{R}{2}}(x_0)\cap\hat\Sigma_{\infty}$. In particular, $x_0$ is a regular point of $\hat\Sigma_{\infty}$. Furthermore, we have $\hat{\textbf{A}}_{\infty}\equiv0$ on $\hat\Sigma_{\infty}\backslash\{0\}$. Therefore, $\hat\Sigma_{\infty}$ consists of a finite union of plane intersecting at the origin.

It remains to show that for each connected component $\Omega_{k}$ of $\Sigma$ with $p_{\gamma}\in{\mathcal S}\cap \bar\Omega_{k}$, $\Omega_k\cup \{p_{\gamma}\}$ is a smooth surface of $M$. In other word, $p_{\gamma}$ is a removable singularity of $\Omega_k$. We essentially follow the argument in \cite{CS} (see also \cite{CM}) dealing with the case of codimension one. By (\ref{E4.3}), we know that $|\textbf{A}|^2\in L^1(\Sigma)$. The absolute continuity of integral implies that
 \begin{equation*}
\lim_{r\to 0}\int_{B_r(p_{\gamma})\cap \Omega_k}|\textbf{A}|^2=0
\end{equation*}
Therefore, given any $0<\delta<1$, there exists some $0<r_{\gamma}<r_0$ with
 \begin{equation*}
\int_{B_{2r_{\gamma}}\cap \Omega_k}|\textbf{A}|^2<\delta\varepsilon.
\end{equation*}
By a variation of Small Regularity Theorem (Theorem \ref{thm-small}), we can easily see that if $r<r_{\gamma}$ and $z\in B_r(p_{\gamma})\backslash B_{\frac{r}{2}}(p_{\gamma})$, then
 \begin{equation}\label{e5.10}
r^2|\textbf{A}|^2(z)<\delta.
\end{equation}

Now we fix $r<r_{\gamma}$ and $z_1\in \partial B_{\frac{3r}{4}}(p_{\gamma})$ It is not hard to see that for $\delta$ sufficiently small, the component of $B_{\frac{r}{4}}(z_1)\cap\Sigma$ containing $z_1$ is a graph with gradient bounded by $C\sqrt{\delta}$ over $T_z\Sigma$. If we repeat this argument for some $z_2\in \partial B_{\frac{3r}{4}}(p_{\gamma})\cap\partial B_{\frac{r}{8}}(z_1)$ in this graph, then we see that the connected component of $\partial B_{\frac{r}{4}}(z_2)\cap \Sigma$ containing $z_2$ is also a graph with bounded gradient. The area bound (\ref{area}) implies that after iterating this argument around $\partial B_{\frac{3r}{4}}(p_{\gamma})$ finite times we must close up. Taking $\delta$ sufficiently small, we see that the connected component of $\left(B_r(p_{\gamma})\backslash B_{\frac{r}{2}}(p_{\gamma}\right)\cap\Sigma$ containing $z_1$ is a graph over a fixed tangent plane with gradient bounded by $C\sqrt{\delta}$. By the estimate of the Hessian ($(29)$ in \cite{Ilm}), $|D^2u|\leq C(1+|Du|^3)|{\bf A}|,$ we know that the Hessian is bounded by $2\sqrt\delta r^{-1}$.

With $\delta>0$ small and $r_{\gamma}>0$ as above, let $\Omega_k$ be a component of $B_{p_{\gamma}}\cap\Sigma$ with $p_{\gamma}\in \bar\Omega_k$. The above discussion shows that $\Omega_k$ is a graph of a vector-valued function $u$ over a fixed tangent plane with $|\nabla u|\leq C\sqrt{\delta}$. We next show that $\nabla u$ has limit at $p_{\gamma}$. To see this, note that for any $\delta_c$, we may argue as above to find some $0<r_c<r_{\gamma}$ such that for any $r<r_{c}$ we have for $z\in \partial B_r(p_{\gamma})\cap\Omega_k$
\begin{equation*}
r^2|Hess_u|^2\leq 4r^2|\textbf A|^2(z)\leq 16\delta_c.
\end{equation*}
Integrating this around $\partial B_r(p_{\gamma})$, and using the fact that $\partial B_r(p_{\gamma})\cap\Omega_k$ is graphical with bounded gradient, we see that
\begin{equation*}
\sup_{z_1,z_2\in \partial B_r(p_{\gamma})\cap \Omega_k}|\nabla u(z_1)-\nabla u(z_2)|\leq 16\pi\sqrt\delta_c.
\end{equation*}
It follows immediately that $\nabla u$ has a limit at $p_{\gamma}$. In particular, $u$ can be extended to a $C^1$ solution of the $\beta_0$-symplectic critical surface equation with uniformly small gradient.  From the equations (\ref{E3.9}) and (\ref{E3.10}) we know that  $u$ satisfies the strictly elliptic equation and the coefficients belong to $C^0(\bar\Omega_k)$. Thus By theorem 9.15 in \cite{GT} we get that $u\in W^{2, p}(\bar\Omega_k)$, for any $p>1$. Using Sobolev embedding theorem and Schauder interior estimates we get $u\in C^\infty(\bar\Omega_k)$. Therefore we can conclude that $\Omega_k\cup \{p_{\gamma}\}$ is a smooth $\beta_0$-symplectic critical surface.
\hfill Q.E.D.

\vspace{.1in}

\section{Appendix-Monotonicity Formula}

In this appendix, we will prove the following monotonicity formula for submanifolds in a Riemannian manifold, which has been used in the proof of the main theorem. It is known to experts. For the convenience of the readers, we provide all the details here.

\begin{theorem}\label{thm-monotonicity}
Let $(M^n,\bar g)$ be a Riemannian manifold with sectional curvature bounded by $K_0$ (i.e., $|K_M|\leq K_0$) and injectivity radius bounded below by $i_0>0$. Suppose that $\Sigma^k\subset M^n$ is a smooth submanifold and $x_0\in M$, if $f$ is a nonnegative function on $\Sigma$, then for any $0<s_1<s_2< \min\{i_0,\frac{1}{\sqrt{K_0}}\}$
\begin{eqnarray}\label{E3.4}
& & e^{k\sqrt{K_0}s_2}\frac{\int_{B_{s_2}(x_0)\cap\Sigma}fd\mu}{s_2^k}-e^{k\sqrt{K_0}s_1}\frac{\int_{B_{s_1}(x_0)\cap\Sigma}fd\mu}{s_1^k}\nonumber\\
& \geq & \int_{(B_{s_2}(x_0)\backslash B_{s_1}(x_0))\cap \Sigma}e^{k\sqrt{K_0}r}\frac{|\nabla^{\perp}r|^2}{r^k}f d\mu\nonumber\\
& & +\int_{s_1}^{s_2}e^{k\sqrt{K_0}s}\frac{1}{2s^{k+1}}\int_{B_s(x_0)\cap \Sigma} {\textbf{H}(r^2)}fd\mu ds\nonumber\\
& & +\int_{s_1}^{s_2}e^{k\sqrt{K_0}s}\frac{1}{2s^{k+1}}\int_{B_s(x_0)\cap \Sigma} (s^2-r^2)\Delta fd\mu ds.
\end{eqnarray}
In particular, by taking $f\equiv1$, we have:
\begin{eqnarray}\label{E3.5}
e^{k\sqrt{K_0}s_2}\frac{Area(B_{s_2}(x_0)\cap\Sigma)}{s_2^k}
& \geq & e^{k\sqrt{K_0}s_1}\frac{Area(B_{s_1}(x_0)\cap\Sigma)}{s_1^k}\nonumber\\
& & +\int_{(B_{s_2}(x_0)\backslash B_{s_1}(x_0))\cap \Sigma}e^{k\sqrt{K_0}r}\frac{|\nabla^{\perp}r|^2}{r^k}d\mu\nonumber\\
& & +\int_{s_1}^{s_2}e^{k\sqrt{K_0}s}\frac{1}{2s^{k+1}}\int_{B_s(x_0)\cap \Sigma} {\textbf{H}(r^2)}d\mu ds.
\end{eqnarray}
\end{theorem}

The proof of the monotonicity formula needs the following estimate, which is a consequence of the standard hessian comparison theorem (see Lemma 5.1 of \cite{CM}).

\begin{lemma}\label{lem-hessian}
Let $(M^n,\bar g)$ be a complete $n$-dimensional Riemannian manifold with sectional curvature bounded by $K_0$ (i.e., $|K_M|\leq K_0$) and injectivity radius bounded below by $i_0>0$. Then for $r<\min\{i_0,\frac{1}{\sqrt{K_0}}\}$ and any vector $X$ with $|X|=1$,
\begin{equation*}
    \left|Hess_r(X,X)-\frac{1}{r}\langle X-\langle X,Dr\rangle Dr,X-\langle X,Dr\rangle Dr\rangle\right|\leq\sqrt{K_0},
\end{equation*}
where $r$ is the distance function from a fixed point on $M$.
\end{lemma}

\textbf{Proof:} Since $-K_0\leq K_M\leq K_0$, by Hessian Comparison Theorem, we have for any $|X|=1$, $X_1\in T\bar M(K_0)$, $X_2\in T\bar M(-K_0)$ with
\begin{equation*}
|X_1|_1=|X_2|_2=|X|=1, \ \ \langle X,Dr\rangle_M=\langle X_1,Dr_1\rangle_{\bar M(K_0)}=\langle X_2,Dr_2\rangle_{\bar M(-K_0)},
\end{equation*}
we have
\begin{equation}\label{e4.3}
    Hess^1_{r_1}(X_1,X_1)\leq Hess_r(X,X)\leq Hess^2_{r_2}(X_2,X_2).
\end{equation}
Here, $\bar M(K_0)$, $\bar M(K_0)$ denote the space form of constant curvature $K_0$, $-K_0$, $r_1$ and $r_2$ are the distance functions on $\bar M(K_0)$ and $\bar M(-K_0)$, and $Hess^1$ and $Hess^2$ are the Hessians of $\bar M(K_0)$ and $\bar M(-K_0)$, respectively.

Recall that, if we set
\begin{equation*}
f(r)=\begin{cases}
            \frac{\sin(\sqrt{K_0}r)}{\sqrt{K_0}}, \ \ \ \ \ \ \ \ K_0>0, \ r<\frac{\pi}{\sqrt{K_0}}, \\
            r, \ \ \ \ \ \ \ \ \ \ \ \ \ \ \ \ \  K_0=0, \\
            \frac{\sinh(\sqrt{-K_0}r)}{\sqrt{-K_0}}, \ \ \ \ \ K_0<0,
       \end{cases}
\end{equation*}
then the hessian of the distance function $r$ on the space form $\bar M(K_0)$ is given by
\begin{equation*}
Hess_r=\frac{f'(r)}{f(r)}(\bar g-dr\otimes dr).
\end{equation*}
Therefore, we have
\begin{eqnarray*}
Hess^1_{r_1}(X_1,X_1)& = & \frac{\sqrt{K_0}\cos(\sqrt{K_0}r)}{\sin(\sqrt{K_0}r)}(|X_1|_1^2-\langle X_1,Dr_1\rangle_{\bar M(K_0)}\langle X_1,Dr_1\rangle_{\bar M(K_0)})\\
& = &  \frac{\sqrt{K_0}\cos(\sqrt{K_0}r)}{\sin(\sqrt{K_0}r)}(|X|^2-\langle X,Dr\rangle\langle X,Dr\rangle)\\
 & = & \frac{\sqrt{K_0}\cos(\sqrt{K_0}r)}{\sin(\sqrt{K_0}r)}\langle X-\langle X,Dr\rangle Dr,X-\langle X,Dr\rangle Dr\rangle.
\end{eqnarray*}
Similarly, we have
\begin{equation*}
Hess^1_{r_2}(X_2,X_2)=\frac{\sqrt{K_0}\cosh(\sqrt{K_0}r)}{\sinh(\sqrt{K_0}r)}\langle X-\langle X,Dr\rangle Dr,X-\langle X,Dr\rangle Dr\rangle.
\end{equation*}
Then, by (\ref{e4.3}), we can easily see that
\begin{eqnarray*}
& & \frac{\sqrt{K_0}r\cos(\sqrt{K_0}r)-\sin(\sqrt{K_0}r)}{\sqrt{K_0}r\sin(\sqrt{K_0}r)} \langle X-\langle X,Dr\rangle Dr,X-\langle X,Dr\rangle Dr\rangle\sqrt{K_0}\\
& \leq&  Hess_r(X,X)-\frac{1}{r}\langle X-\langle X,Dr\rangle Dr,X-\langle X,Dr\rangle Dr\rangle\\
 & \leq& \frac{\sqrt{K_0}r\cosh(\sqrt{K_0}r)-\sinh(\sqrt{K_0}r)}{\sqrt{K_0}r\sinh(\sqrt{K_0}r)} \langle X-\langle X,Dr\rangle Dr,X-\langle X,Dr\rangle Dr\rangle\sqrt{K_0}.
\end{eqnarray*}
Note that
\begin{equation*}
0\leq \langle X-\langle X,Dr\rangle Dr,X-\langle X,Dr\rangle Dr\rangle=|X|^2-\langle X,Dr\rangle^2\leq1.
\end{equation*}
An elementary computation shows that
\begin{equation*}
\frac{x\cosh x-\sinh x}{x\sinh x}\leq 1, \ \ \forall x>0; \ \ \
\frac{x\cos x-\sin x}{x\sin x}\geq -1, \ \ \forall x\in (0,1].
\end{equation*}
Then the conclusion follows easily.
\hfill Q.E.D.

\vspace{.1in}

\textbf{Proof of Theorem \ref{thm-monotonicity}:} In the proof, we will denote $r$ the extrinsic distance from a fixed point $x_0$ on $M$, $Dr$ the gradient of $r$ with respect to the Rimammian metric $\bar g$. Then $Dr=\nabla r+\nabla^{\perp}r$, where $\nabla r$ and $\nabla^{\perp} r$ denote the projections of $Dr$ on the tangent bundle $T\Sigma$ and normal bundle $N\Sigma$, respectively. Also, we denote $D$ and $\nabla$ the Levi-Civita connection on $(M,\bar g)$ and the induced connection on $(\Sigma,g)$, respectively, where $g$ denotes the induced metric on $\Sigma$ from $(M,\bar g)$.

We choose local orthonormal frame $\{e_1,\cdots,e_k,N_1,\cdots,N_{n-k}\}$ so that: $\{e_1,\cdots,e_k\}$ spans $T\Sigma$ and $\{N_1,\cdots,N_{n-k}\}$ spans $N\Sigma$. Then for any $C^2$ function $f$ on $\Sigma$, we have
\begin{eqnarray*}
\Delta f
& = &  \sum_{i=1}^{k}[e_i(e_i(f))-(\nabla_{e_i}e_i) f]\\
& = &  \sum_{i=1}^{k}[e_i(e_i(f))-(D_{e_i}e_i) f+\textbf{A}(e_i,e_i)f]\\
& = &  \sum_{i=1}^{k}Hess_f(e_i,e_i)+\textbf{H}(f),
\end{eqnarray*}
where ${\textbf A}$ denotes the second fundamental form of $\Sigma$ in $M$, $\textbf{H}$ is the mean curvature vector of $\Sigma$ in $M$, and $Hess$ is the Hessian of $(M,\bar g)$. Therefore, we have
\begin{eqnarray*}
\Delta r^2
& = &  \sum_{i=1}^{k}Hess_{r^2}(e_i,e_i)+\textbf{H}(r^2)\\
& = &  \sum_{i=1}^{k}[2rHess_r(e_i,e_i)+2(dr\otimes dr)(e_i,e_i)]+\textbf{H}(r^2)\\
& = &  \sum_{i=1}^{k}[2rHess_r(e_i,e_i)+2\langle e_i,Dr\rangle\langle e_i,Dr\rangle]+\textbf{H}(r^2)\\
& = &  2r\sum_{i=1}^{k}[Hess_r(e_i,e_i)+\frac{1}{r}\langle e_i,Dr\rangle\langle e_i,Dr\rangle-\frac{1}{r}\langle e_i,e_i\rangle+\frac{1}{r}]+\textbf{H}(r^2)\\
& = &  2r\sum_{i=1}^{k}[Hess_r(e_i,e_i)-\frac{1}{r}\langle e_i-\langle e_i,Dr\rangle Dr,e_i-\langle e_i,Dr\rangle Dr\rangle]+2k+\textbf{H}(r^2).
\end{eqnarray*}
Applying Lemma \ref{lem-hessian}, we have
\begin{equation}\label{e4.4}
   |\Delta r^2-2k-\textbf{H}(r^2)|\leq 2k\sqrt{K_0}r.
\end{equation}
Denote $B_s(x_0)$ the geodesic ball of radius $s$ centered at $x_0$ on $M$. For $x_0\in M$,we see that the unit outer normal vector of $\partial B_s(x_0)$ is given by $Dr$ and the unit outer normal $\partial B_s(x_0)\cap\Sigma$ is given by $\frac{\nabla r}{|\nabla r|}$.

The coarea formula gives us that
\begin{equation}\label{e4.5}
   \int_{B_s(x_0)\cap \Sigma} fd\mu=\int_0^s\left(\int_{\partial B_{\tau}(x_0)\cap \Sigma}\frac{f}{|\nabla r|}d\sigma\right)d\tau,
\end{equation}
which implies that
\begin{equation}\label{e4.6}
   \frac{d}{ds}\int_{B_s(x_0)\cap \Sigma} fd\mu=\int_{\partial B_{s}(x_0)\cap \Sigma}\frac{f}{|\nabla r|}d\sigma.
\end{equation}
On the other hand, by divergence theorem, we have
\begin{eqnarray*}
\int_{B_s(x_0)\cap \Sigma} f\Delta r^2d\mu
& = & \int_{\partial B_s(x_0)\cap \Sigma} f\frac{\partial r^2}{\partial \nu}d\mu-\int_{B_s(x_0)\cap \Sigma} \langle\nabla f,\nabla r^2\rangle d\mu\\
& = &  \int_{\partial B_s(x_0)\cap \Sigma} 2rf\langle \nabla r,\frac{\nabla r}{|\nabla r|}\rangle d\mu-\int_{\partial B_s(x_0)\cap \Sigma} \frac{\partial f}{\partial \nu}r^2d\sigma+\int_{B_s(x_0)\cap \Sigma}r^2\Delta fd\mu\\
& = &   \int_{\partial B_s(x_0)\cap \Sigma} 2rf|\nabla r|d\mu-s^2\int_{B_s(x_0)\cap \Sigma} \Delta fd\mu+\int_{B_s(x_0)\cap \Sigma}r^2\Delta fd\mu\\
& = &   2s\int_{\partial B_s(x_0)\cap \Sigma} f|\nabla r|d\mu-\int_{B_s(x_0)\cap \Sigma} (s^2-r^2)\Delta fd\mu.
\end{eqnarray*}
Combining with (\ref{e4.4}) and (\ref{e4.6}), we get that
\begin{eqnarray*}
\frac{d}{ds}\left(\frac{\int_{B_s(x_0)\cap\Sigma}fd\mu}{s^k}\right)
& = & \frac{1}{s^{k+1}}\left(s\int_{\partial B_{s}(x_0)\cap \Sigma}\frac{f}{|\nabla r|}d\sigma-k\int_{B_s(x_0)\cap\Sigma}fd\mu\right)\\
& = &   \frac{1}{s^{k+1}}\left(s\int_{\partial B_{s}(x_0)\cap \Sigma}\frac{f}{|\nabla r|}d\sigma-s\int_{\partial B_s(x_0)\cap \Sigma} f|\nabla r|d\mu\right.\\
& & \left.+\int_{B_s(x_0)\cap \Sigma} f(\frac{1}{2}\Delta r^2-k)d\mu+\frac{1}{2}\int_{B_s(x_0)\cap \Sigma} (s^2-r^2)\Delta fd\mu \right)\\
& \geq&   \frac{1}{s^{k}}\int_{\partial B_{s}(x_0)\cap \Sigma}\frac{|\nabla^{\perp}r|^2}{|\nabla r|}f d\sigma+\frac{1}{2s^{k+1}}\int_{B_s(x_0)\cap \Sigma} {\textbf{H}(r^2)}fd\mu\\
& & -k\sqrt{K_0}\frac{\int_{B_s(x_0)\cap \Sigma}fd\mu}{s^k}+\frac{1}{2s^{k+1}}\int_{B_s(x_0)\cap \Sigma} (s^2-r^2)\Delta fd\mu,
\end{eqnarray*}
which implies that
\begin{eqnarray*}
\frac{d}{ds}\left(e^{k\sqrt{K_0}s}\frac{\int_{B_s(x_0)\cap\Sigma}fd\mu}{s^k}\right)
& \geq&   e^{k\sqrt{K_0}s}\frac{1}{s^{k}}\int_{\partial B_{s}(x_0)\cap \Sigma}\frac{|\nabla^{\perp}r|^2}{|\nabla r|}f d\sigma\\
& & +e^{k\sqrt{K_0}s}\frac{1}{2s^{k+1}}\int_{B_s(x_0)\cap \Sigma} {\textbf{H}(r^2)}fd\mu\\
& & +e^{k\sqrt{K_0}s}\frac{1}{2s^{k+1}}\int_{B_s(x_0)\cap \Sigma} (s^2-r^2)\Delta fd\mu\\
& = &  \frac{d}{ds} \left(\int_{B_{s}(x_0)\cap \Sigma}e^{k\sqrt{K_0}r}\frac{|\nabla^{\perp}r|^2}{r^k}f d\sigma\right)\\
& & +e^{k\sqrt{K_0}s}\frac{1}{2s^{k+1}}\int_{B_s(x_0)\cap \Sigma} {\textbf{H}(r^2)}fd\mu\\
& & +e^{k\sqrt{K_0}s}\frac{1}{2s^{k+1}}\int_{B_s(x_0)\cap \Sigma} (s^2-r^2)\Delta fd\mu.
\end{eqnarray*}
Integration from $s_1$ to $s_2$ gives (\ref{E3.4}).
\hfill Q.E.D.

\vspace{.1in}

The following corollary is a generalization of (1.3) in \cite{Simon}.

\begin{corollary}\label{cor-monotonicity-2}
Let $(M^n,\bar g)$ be a closed Riemannian manifold with sectional curvature bounded by $K_0$ (i.e., $|K_M|\leq K_0$) and injectivity radius bounded below by $i_0>0$. Suppose that $\Sigma^2\subset M^n$ is a smooth surface and $x_0\in M$, then for any $0<s_1<s_2< \min\{i_0,\frac{1}{\sqrt{K_0}}\}$
\begin{eqnarray}\label{E3.6}
&  & e^{4\sqrt{K_0}s_1}\left(\frac{Area(B_{s_1}(x_0)\cap\Sigma)}{s_1^2}+\int_{(B_{s_2}(x_0)\backslash B_{s_1}(x_0))\cap \Sigma}\frac{|\nabla^{\perp}r|^2}{r^2}d\mu\right)\nonumber\\
& \leq & 2e^{4\sqrt{K_0}s_2}\left(\frac{Area(B_{s_2}(x_0)\cap\Sigma)}{s_2^2}+\int_{B_{s_2}(x_0)\cap\Sigma}|\textbf{H}|^2d\mu\right).
\end{eqnarray}
\end{corollary}

\textbf{Proof:} First note that (\ref{E3.5}) (with $k=2$) can be written as
\begin{eqnarray}\label{E3.7}
e^{2\sqrt{K_0}s_1}\frac{Area(B_{s_1}(x_0)\cap\Sigma)}{s_1^2}
& \leq & e^{2\sqrt{K_0}s_2}\frac{Area(B_{s_2}(x_0)\cap\Sigma)}{s_2^2}\nonumber\\
& & -\int_{(B_{s_2}(x_0)\backslash B_{s_1}(x_0))\cap \Sigma}e^{2\sqrt{K_0}r}\frac{|\nabla^{\perp}r|^2}{r^2}d\mu\nonumber\\
& & -\int_{s_1}^{s_2}e^{2\sqrt{K_0}s}\frac{1}{2s^{3}}\int_{B_s(x_0)\cap \Sigma} {\textbf{H}(r^2)}d\mu ds\nonumber\\
& \leq & e^{2\sqrt{K_0}s_2}\frac{Area(B_{s_2}(x_0)\cap\Sigma)}{s_2^2}\nonumber\\
& & -e^{2\sqrt{K_0}s_1}\int_{(B_{s_2}(x_0)\backslash B_{s_1}(x_0))\cap \Sigma}\frac{|\nabla^{\perp}r|^2}{r^2}d\mu\nonumber\\
& & -\int_{s_1}^{s_2}e^{2\sqrt{K_0}s}\frac{1}{2s^{3}}\int_{B_s(x_0)\cap \Sigma} {\textbf{H}(r^2)}d\mu ds.
\end{eqnarray}
We need to estimate the last term. By coarea formula and Fubini Theorem, we compute
\begin{eqnarray*}
& & \int_{s_1}^{s_2}e^{2\sqrt{K_0}s}\frac{1}{2s^{3}}\int_{B_s(x_0)\cap \Sigma} {\textbf{H}(r^2)}d\mu ds\\
& = & \int_{s_1}^{s_2}\int_{0}^{s}e^{2\sqrt{K_0}s}\frac{1}{2s^{3}}\int_{(\partial B_{\tau}(x_0))\cap \Sigma} \frac{\textbf{H}(r^2)}{|\nabla r|}d\sigma d\tau ds\\
& = &  \int^{s_1}_{0}\int_{s_1}^{s_2}e^{2\sqrt{K_0}s}\frac{1}{2s^{3}}\int_{(\partial B_{\tau}(x_0))\cap \Sigma} \frac{\textbf{H}(r^2)}{|\nabla r|}d\sigma ds d\tau\\
& & +\int_{s_1}^{s_2}\int_{\tau}^{s_2}e^{2\sqrt{K_0}s}\frac{1}{2s^{3}}\int_{(\partial B_{\tau}(x_0))\cap \Sigma} \frac{\textbf{H}(r^2)}{|\nabla r|}d\sigma ds d\tau\\
& = &  \int_{s_1}^{s_2}e^{2\sqrt{K_0}s}\frac{1}{2s^{3}}ds\int_{B_{s_1}(x_0)\cap \Sigma} \textbf{H}(r^2)d\mu\\
& & +\int_{(B_{s_2}(x_0)\backslash B_{s_1}(x_0))\cap \Sigma}\left(\int_{r}^{s_2}e^{2\sqrt{K_0}s}\frac{1}{2s^{3}}ds\right) \textbf{H}(r^2)d\mu.
\end{eqnarray*}
Since $\textbf{H}(r^2)=2r\langle \textbf{H},\nabla^{\perp}r\rangle$, we have
\begin{eqnarray*}
& & \left|-\int_{s_1}^{s_2}e^{2\sqrt{K_0}s}\frac{1}{2s^{3}}\int_{B_s(x_0)\cap \Sigma} {\textbf{H}(r^2)}d\mu ds\right|\\
& \leq &e^{2\sqrt{K_0}s_2}\int_{s_1}^{s_2}\frac{1}{s^{3}}ds\int_{B_{s_1}(x_0)\cap \Sigma} r|\textbf{H}||\nabla^{\perp}r|d\mu\\
& & +e^{2\sqrt{K_0}s_2}\int_{(B_{s_2}(x_0)\backslash B_{s_1}(x_0))\cap \Sigma}\left(\int_{r}^{s_2}\frac{1}{s^{3}}ds\right) r|\textbf{H}||\nabla^{\perp}r|d\mu\\
& = &e^{2\sqrt{K_0}s_2}\frac{1}{2s_1^{2}}\int_{B_{s_1}(x_0)\cap \Sigma} r|\textbf{H}||\nabla^{\perp}r|d\mu-e^{2\sqrt{K_0}s_2}\frac{1}{2s_2^{2}}\int_{B_{s_2}(x_0)\cap \Sigma} r|\textbf{H}||\nabla^{\perp}r|d\mu\\
& & +\frac{1}{2}e^{2\sqrt{K_0}s_2}\int_{(B_{s_2}(x_0)\backslash B_{s_1}(x_0))\cap \Sigma}\frac{1}{r}|\textbf{H}||\nabla^{\perp}r|d\mu\\
& \leq & e^{2\sqrt{K_0}s_2}\frac{1}{2s_1}\int_{B_{s_1}(x_0)\cap \Sigma} |\textbf{H}|d\mu+\frac{1}{2}e^{2\sqrt{K_0}s_2}\int_{(B_{s_2}(x_0)\backslash B_{s_1}(x_0))\cap \Sigma}\frac{1}{r}|\textbf{H}||\nabla^{\perp}r|d\mu\\
& \leq & \frac{1}{2}e^{2\sqrt{K_0}s_1}\frac{Area(B_{s_1}(x_0)\cap\Sigma)}{s_1^2}+\frac{1}{8}e^{4\sqrt{K_0}s_2-2\sqrt{K_0}s_1}\int_{B_{s_1}(x_0)\cap \Sigma} |\textbf{H}|^2d\mu\\
& & +\frac{1}{2}e^{2\sqrt{K_0}s_1}\int_{(B_{s_2}(x_0)\backslash B_{s_1}(x_0))\cap \Sigma}\frac{|\nabla^{\perp}r|^2}{r^2}d\mu+\frac{1}{8}e^{4\sqrt{K_0}s_2-2\sqrt{K_0}s_1}\int_{(B_{s_2}(x_0)\backslash B_{s_1}(x_0))\cap \Sigma}|\textbf{H}|^2d\mu\\
& = & \frac{1}{2}e^{2\sqrt{K_0}s_1}\frac{Area(B_{s_1}(x_0)\cap\Sigma)}{s_1^2}+\frac{1}{2}e^{2\sqrt{K_0}s_1}\int_{(B_{s_2}(x_0)\backslash B_{s_1}(x_0))\cap \Sigma}\frac{|\nabla^{\perp}r|^2}{r^2}d\mu\\
& & +\frac{1}{8}e^{4\sqrt{K_0}s_2-2\sqrt{K_0}s_1}\int_{B_{s_2}(x_0)\cap \Sigma}|\textbf{H}|^2d\mu.
\end{eqnarray*}
Putting this inequality into (\ref{E3.7}) yields the desired estimate.
\hfill Q.E.D.

\end{document}